\documentclass{amsart}
\usepackage{amsmath,amsthm,amssymb}
\usepackage{graphicx}
\usepackage{tikz-cd}
\usepackage{paralist}
\usepackage{environ}
\usepackage{xcolor}
\usepackage{hyperref}
\usepackage{centernot}
\usepackage{mathtools}
\usepackage{marvosym}
\usepackage{xcolor,cancel}
\usepackage{float}
\usepackage{subcaption}

\usepackage{verbatim}
\usepackage{csquotes}
\usepackage{pstricks}
\usepackage{mdframed}
\usepackage[most]{tcolorbox}

\makeatletter
\def\square{\pst@object{square}}
\def\square@i(#1,#2)#3{{\use@par\solid@star\psframe[origin={#1,#2}](#3,#3)}}
\makeatother

\makeatletter
\DeclareFontFamily{U}{tipa}{}
\DeclareFontShape{U}{tipa}{bx}{n}{<->tipabx10}{}
\newcommand{\arc@char}{{\usefont{U}{tipa}{bx}{n}\symbol{62}}}%

\newcommand{\arc}[1]{\mathpalette\arc@arc{#1}}

\newcommand{\arc@arc}[2]{%
	\sbox0{$\m@th#1#2$}%
	\vbox{
		\hbox{\resizebox{\wd0}{\height}{\arc@char}}
		\nointerlineskip
		\box0
	}%
}
\makeatother

\makeatletter
\newcommand{\doublewedge}{\big@doubleop{\wedge}}
\newcommand{\big@doubleop}[1]{%
	\DOTSB\mathop{\mathpalette\big@doubleop@aux{#1}}\slimits@
}

\newcommand\big@doubleop@aux[2]{%
	\sbox\z@{$\m@th#1#2$}%
	\makebox[1.35\wd\z@][s]{$\m@th#1#2\hss#2$}%
}

\newcommand{\dnear}{\delta_{\Phi}} 

\theoremstyle{plain}
\newtheorem{theorem}{Theorem}
\newtheorem{lemma}{Lemma}

\newtheorem{definition}{Definition}

\newtheorem{corollary}{Corollary}
\newtheorem{proposition}{Proposition}

\usepackage{cancel}
 
\usepackage[title]{appendix}

\begin{document}
	\title{ On approximately ideals: primary and 1-absorbing ideals in descriptive relator spaces}
	\author[M. Almahariq]{Maram Almahariq}
	\address{Department of Mathematics, Birzeit university, Ramallah, Palestine,
	}
	\email{maram.mahareeq.14@gmail.com}
    

\begin{abstract}
For the \(\textit{approx.}\) ideal $\mathcal{W}$ of the \(\textit{approx.}\)  commutative ring ${\mathcal{R}}$ with unity in a descriptive relator space, after introducing the \(\textit{approx.}\) prime ideal in \cite{maram_first_article}, this work demonstrates some special properties of the \(\textit{approx.}\)  ideals-specifically, the \(\textit{approx.}\)  primary ideal, the \(\textit{approx.}\) semi-primary ideal and the \(\textit{approx.}\) 1-absorbing primary ideal. A set of theorems related to these concepts is presented. Among them is this important result: If $\mathcal{W}$ is an \(\textit{approx.}\) 1-absorbing primary ideal, then $r(\mathcal{W})$ is an \(\textit{approx.}\) prime ideal of the \(\textit{approx.}\) ring ${\mathcal{R}}$. Furthermore, the relationship between these classes is studied: If $\mathcal{W}$ is an \(\textit{approx.}\) prime ideal of ${\mathcal{R}}$, then $\mathcal{W}$ is also an \(\textit{approx.}\) primary ideal. Moreover, it turns out that this is an \(\textit{approx.}\) 1-absorbing primary ideal.
\end{abstract}
	\maketitle
	\tableofcontents

{
\section{Introduction}

Computational proximity $(CP)$ is an algorithmic approach to finding sets of non-empty points, whether they are close or far apart. $(CP)$  methods for finding close or far apart sets rely on the study of structures called proximity spaces. One of the most basic and important proximity spaces is descriptive proximity spaces, which are essentially based on a probe function $(\Phi)$, that represents the nature of points, such as color, the intensity of a particular color, or the texture of a point \cite{Peters2013mcsintro,Peters2016ComputationalProximity}. This vector defines the general properties of points. If two points share a specific property, I can say that the two points are descriptively close to each other. This can be expressed through the descriptive intersection between two nonempty sets as follows \cite{DiConcilio2018MCSdescriptiveProximities}:
	
	\begin{center}
    \colorbox{green!20} { $ W \mathop{\cap}\limits_{\Phi} Q = \{s\in W \cup Q :  \Phi(s) \in \Phi(W) \cap \Phi(Q)\}$}.
    \end{center}

The axioms formulated by \v{C}ech for spatial proximity spaces \cite{Cech1966} were transferred to axioms specific to descriptive proximity spaces, which are presented as follows \cite{DiConcilio2018MCSdescriptiveProximities, NaimpallyWarrack1970}.

\begin{description}
	\item[{\bf \itshape (DP.0)}]  Any subset of $X$ is far from the void set. 
	\item[{\bf \itshape (DP.1)}] $S \dnear K $ if and only if $K \dnear S$
	\item[{\bf \itshape (DP.2)}] $K  \mathop{\cap}\limits_{\Phi} S \neq \emptyset$ if and only if $K \dnear S$.
	\item[{\bf \itshape (DP.3)}] $K \ \dnear (S\cup L) $ $\iff$ $K \dnear S$ or $K \dnear L$. \qquad 
    \textcolor{blue}{$\blacksquare$}
    \leavevmode\vspace{1.5em}
\end{description} 

The descriptive proximity space is not limited to studying the distance and nearness between two or more sets; rather, the concept extended to the formulation of algebraic structures such as approximately groups, approximately rings, etc. To present these terms, the concept of descriptive proximity space was expanded beyond a single proximity. A set of relations was defined on a nonempty set $X$, and a new algebraic structure known as a descriptive relator space $(X,\mathcal{U}_{\dnear})$ was obtained \cite{Peters2016relator}. 
To define any of these algebraic structures, Inan introduced the concept of the upper approximation of a set, which can be stated as follows \cite{Inan:2017}.\newline

For a subset $N$ of a descriptive relator space  $(X,\mathcal{U}_{\dnear})$, the descriptively upper approximation of $N$, denoted by $\Phi^{*}N$, is defined by 
    \begin{center}	\colorbox{green!20} {$		
	\Phi^{*} N = \{ x \in X : \{x\} \ \dnear  N  \}.$}  \qquad 
    \end{center}

\begin{definition} 
  \cite{Inan:2019}  Let $(X,\mathcal{U}_{\dnear})$ be a descriptive  relator space, $\mathcal {H}$ a subset of $X$ and   \enquote{$\cdot$} a binary operation on $X$. Then, $\mathcal{H}$ is called an \(\textit{approx.}\)  group, provided the following conditions are satisfied. \\
  \begin{description}
	\item[(${\bf  AH}_1$)] $\forall$ $a,b \in \mathcal{H}$, \colorbox{green!20}{$a\cdot b \in \Phi^{*} \mathcal{H}$} 
	\item[(${\bf AH}_2$)] $\forall a,b,c \in \mathcal {H}$,  \colorbox{green!20}{$(a\cdot b)\cdot c = a\cdot(b\cdot c)$} holds in $\Phi^{*}\mathcal{H}$ .
	\item[(${\bf AH}_3$)] There exists $1_\mathcal{H} \in \Phi^{*} \mathcal{H}$ such that \colorbox{green!20}{$a\cdot 1_\mathcal{H} =1_\mathcal{H} \cdot a=a$} for all $a \in \mathcal{H}$. 	
	\item[(${\bf AH}_4$)] There exists $b \in \mathcal{H}$ such that \colorbox{green!20}{$a\cdot b= b\cdot a = 1_\mathcal{H}.$} for all $a\in \mathcal{H}$. 
\end{description}
		\qquad 
\end{definition}
	
Here, I say that $\mathcal{H}$ is an \(\textit{approx.}\) groupoid, if (${\bf AH}_1$) is satisfied, and that $\mathcal{H}$ is an \(\textit{approx.}\)
 semi-group if both 
 (${\bf AH}_1$) and (${\bf AH}_2$) are satisfied. \cite{Inan:2017} 
\begin{definition} \cite{Inan2019TJMdescriptiveProximity}  
Let $(X,\mathcal{U}_{\dnear})$be a descriptive  relator space. For a subset $\mathcal{R}$ of $X$, let \enquote{+} and \enquote{$\cdot$} be two binary operations on $X$, then $\mathcal{R}$ is called an \(\textit{approx.}\) ring, if the following conditions are satisfied.
\begin{enumerate}
	\item[$({\bf AR}_1)$] $(\mathcal{R},+)$ is an \(\textit{approx.}\) abelian group.
	\item[$({\bf AR}_2)$] $(\mathcal{R},\cdot)$ is an  \(\textit{approx.}\) semi-group.
	\item[$({\bf AR}_3)$]  for all $a,b,c \in \mathcal{R}$, the distributive property holds in $\Phi^{*}\mathcal{R}$ .
\end{enumerate}
Moreover,
\begin{enumerate}
  \item[$({\bf AR}_4)$] If the commutative property holds for the multiplication operation, then $\mathcal{R}$ is called an \(\textit{approx.}\) commutative ring. 
  \item[$({\bf AR}_5)$] If  $\Phi^{*} \mathcal{R}$ contains the unity element $1_R$, then $\mathcal{R}$ is called an \(\textit{approx.}\) ring with unity (identity). \qquad 
\end{enumerate} 
\end{definition}



\begin{definition} 
\cite{maram_first_article}  
Let $W$ be an \(\textit{approx.}\) ideal of an \(\textit{approx.}\)  ring $\mathcal{R}$. For $a,b \in {\mathcal{R}}$ and $ab \in \Phi^{*}W$, then $W$ is called an \(\textit{approx.}\) prime ideal, provided either $a \in W$ or $b \in W$. \qquad 
\end{definition}

\begin{definition}  \cite {maram_first_article}
 An \(\textit{approx.}\)  commutative ring $\mathcal{R}$ is called an \(\textit{approx.}\)   integral domain, provided for any $m$, $n$ $\in$ $\mathcal{R}$, if  $m . n$ $\in$ $\Phi^{*}$$\mathcal{R}$ and $m . n =0$, then either $m=0$ or $n=0$. \qquad 
\end{definition}
	

\begin{definition}  \cite{Inan2019TJMdescriptiveProximity} 
Let $Q$ be a non-void subset of an \(\textit{approx.}\) commutative ring $\mathcal {R}$, then $Q$ is called an \(\textit{approx.}\) ideal, provided for any $a,b \in Q$ and $r \in \mathcal{R}$. I have $a + b \in \Phi^{*} Q$, $-a \in Q$ and $r \cdot a \in \Phi^{*} Q$.
\end{definition}

	
In this paper, I have been concerned with defining the properties of approximately  ideals in approximately rings. One such property was introduced in \cite{maram_first_article}, where the definition of approximately prime ideals and their related properties were obtained. Here, I give definitions of approximately primary ideals and approximately 1-absorbing ideals, and discuss their consequences.

\section{Main results}

This work is influenced by previous studies in \cite{Badawi2020, coleman2020primary, gilmer1962rings} where these concepts were originally introduced and studied in the context of fundamental rings. Here, I present an extension of these theories to include \( \textit{approx.}\)  commutative rings with a unity in the descriptive relator space.

\subsection{Approximately primary ideals}

\begin{definition} \label{approx.primary} 
If $W$ is an \( \textit{approx.} \) ideal of an \( \textit{approx.} \) ring ${\mathcal{R}}$, then $W$ is called an \( \textit{approx.} \) primary ideal, provided for $a,b \in {\mathcal{R}}$ and $ab \in \Phi^{*} W$, implies $a \in W$ or $b^{n} \in W$ for some $n \geq 1$.
\end{definition}

Any \( \textit{approx.} \) prime ideal is an \( \textit{approx.} \) primary ideal, as explained below, however the converse does not hold. \newline

\begin{theorem}
Let $W$ be an \( \textit{approx.} \) prime ideal, then $W$ is an \( \textit{approx.} \) primary ideal.
\end{theorem}
\begin{proof}
The proof follows directly form definition ~\ref{approx.primary}.	 
\end{proof}
\begin{definition} 
Let $s$ be an element of an \( \textit{approx.}\) ring ${\mathcal{R}}$, then $s$ is called an \( \textit{approx.}\) nilpotent element, provided for some $m \in N$, $s^m = 0$, where 
$s^m \in \Phi^{*}{\mathcal{R}}$.
\end{definition}

\begin{theorem} 
Suppose that $W$ is an \( \textit{approx.}\) primary ideal in an \( \textit{approx.}\) ring ${\mathcal{R}}$, then every zero divisor in $\mathcal{R}/W$ is an \( \textit{approx.}\) nilpotent. 
\end{theorem} 
\begin{proof}
For an \(\textit{approx.}\) primary ideal $W$, $\mathcal{R}/W$ is a non-trivial \(\textit{approx.}\) ring. Let $s+W$ be a zero divisor in $\mathcal{R}/W$, which leads to the existence of a non-zero element $l+W$ such that $(s+W).(l+W) = sl + W = W$, provided $sl \in \Phi^{*} W$. Since $l \notin W$ and $W$ is an \(\textit{approx.}\) primary ideal. Thus, for some $n \in N$, $s^{n} \in W$. Hence, $s^{n} + W = (s+W)^{n} = W$. Moreover, $s+W$ is an \(\textit{approx.}\) nilpotent. 
\end{proof}

Assume ${\mathcal{R}}$ is an \(\textit{approx.}\) commutative ring with a unity and the pair \colorbox{green!20} {$(\Phi^{*}\mathcal{R},+),(\Phi^{*}\mathcal{R},\cdot)$} are groupoids. Let $W$ be an \(\textit{approx.}\)  ideal in ${\mathcal{R}}$. The set $r(W)$ 
    \begin{center}
    \colorbox{green!20} 
	{$r(W)= \{ s \in \mathcal{R} \ : s^n \in W, \ n\in N  \ \}$}
    \end{center}
is an \(\textit{approx.}\)  radical ideal in ${\mathcal{R}}$.
	
\begin{theorem}
  For an  \(\textit{approx.}\) commutative ring ${\mathcal{R}}$, if $W$ is an \(\textit{approx.}\) ideal of ${\mathcal{R}}$ and $\Phi^{*}\mathcal{R}$ is groupoid with binary operations \enquote{+} and \enquote{$\cdot$}. Then, $r(W)$ is an \(\textit{approx.}\) radical ideal of ${\mathcal{R}}$. 
\end{theorem}
\begin{proof}
  Let $s$ be an element of $r(W)$ and $l \in \mathcal{R}$, then $s^n \in W$ for some $n \in N$. Now, $(ls)^{n} = l^n.s^n$, since $W$ is an \(\textit{approx.}\) ideal, then $l^n.s^n \in \Phi^{*} W$. Thus, $(ls) \in \Phi^{*}  r(W)$.\newline
  Now, assume that $s$ and $k$ are elements of $r(W)$, then $s^n$ and $k^m$ are elements of $W$ for some $n,m \in N$. Then, $(s+k)^{(n+m)}$ is an element of $\Phi^{*}r(W)$, using the binomial theorem. In addition, if $s \in r(W)$, then $-s$ is also an element of $r(W)$ since $W$ is an \(\textit{approx.}\) ideal.
\end{proof}

An approximately radical ideal is not merely an approximately ideal, under certain conditions, as shown in the following proposition, it is the smallest approximately prime ideal containing a certain ideal.

\begin{proposition} \label{smallest}
Let $W$ be an \(\textit{approx.}\) primary ideal of an \(\textit{approx.}\) ring ${\mathcal{R}}$. Then, the \(\textit{approx.}\) radical $r(W)$ of $W$ is the smallest \(\textit{approx.}\) prime ideal containing $W$.
\end{proposition}

\begin{proof}
Suppose $s,k$ are elements of ${\mathcal{R}}$, and $sk \in \Phi^{*}r(W)$. Then there exists a positive integer $n$ such that $(sk)^{n}$ $= s^{n} k^{n}$ $\in$ $\Phi^{*}W$. Since $W$ is an \(\textit{approx.}\)  primary ideal, it follows that either $s^n$ $\in$ $W$ or $k^{nm}$ $\in$ $W$ for some $m \in N$. Now, from the definition of the \(\textit{approx.}\) radical ideal. I have either $s$ in $r(W)$ or $k$ in $r(W)$. Thus, $r(W)$ is an \(\textit{approx.}\) prime ideal.\newline\newline
Now, to show that $r(W)$ is the smallest \(\textit{approx.}\) prime ideal, let $P$ be an \(\textit{approx.}\) prime ideal such that $W$ is a subset of $P$. If $s$ is an element of $r(W)$, then $s^n$ is an element of $W$ for some $n \in N$. Since $W$ is a subset of $P$, then $s^n$ is an element of $P$. As $P$ is an \(\textit{approx.}\)  prime ideal, this implies $s \in P$. Furthermore, $r(W)$ $\subseteq$ $P$. Hence, $r(W)$ is the smallest \(\textit{approx.}\) prime ideal containing $W$.
\end{proof}

\begin{corollary}
Let $W$ be an \(\textit{approx.}\)  prime ideal, then $r(W)=W$ 
\end{corollary}
\begin{proof}
Notice that $W \subseteq r(W)$. Now, if $s \in r(W)$, then for some $k \in N$. I have $s^{k} \in W$. As $W$ is an \(\textit{approx.}\)  prime ideal, then $s \in W$. Hence, I can conclude that $r(W)=W$.
\end{proof}

\begin{definition} \label{p-primary}
For an \(\textit{approx.}\) primary ideal $W$ and an \(\textit{approx.}\) prime ideal $P$. If $r(W)=P$, then I say that $W$ is an \(\textit{approx.}\) P-primary ideal. 
\end{definition}

\begin{lemma} 
Let $\mathcal{A}$ be a collection of finite \(\textit{approx.}\) ideals $W_{1},W_{2},W_{3},.. W_{n}$ of an \(\textit{approx.}\) ring ${\mathcal{R}}$. If $W = \bigcap_{i=1}^{n} W_i$, then $r(W) = \bigcap_{i=1}^{n} r(W_i)$
\end{lemma}

\begin{proof}
To show that $r(W) \subseteq \bigcap_{i=1}^{n} r(W_i)$, let $s$ be an element of $r(W)$. Then $s^m \in W$ for some $m \in N$, which means that $s^m$ is an element of $W_i$ for all $i \in \{1,2,...n\}$. Thus, $s$ is an element of $r(W_i)$, $\forall i \in \{1,2,...n\}$, and hence $s \in  \bigcap_{i=1}^{n} r(W_i)$.
\newline\newline 
Now, for $\bigcap_{i=1}^{n} r(W_i) \subseteq r(W)$. Let $s$ be an element of $\bigcap_{i=1}^{n} r(W_i)$. Then,  for each $ i \in \{1,2,...n\}$, there exists $m_i \in N$ such that $s^{m_{i}} \in W_i$, $\forall i \in \{1,2,...n\}$. Set $m= max\{{m_i}|  i \in \{1,2,...n\}\}$. Then, $s^m \in W$, and hence $s \in r(W)$.  
\end{proof}

\begin{theorem}
Suppose $P$ is an \(\textit{approx.}\) prime ideal of the \(\textit{approx.}\) ring ${\mathcal{R}}$, and $W_i$, $\forall i \in \{1,2,...,n\}$  are \(\textit{approx.}\) P- primary ideals. If $\Phi^{*} W = \bigcap_{i=1}^{n} \Phi^{*}W_i$, then $W = \bigcap_{i=1}^{n} W_i$ is also an \(\textit{approx.}\) P- primary ideal.
\end{theorem}

\begin{proof}
 $W_i$ are \(\textit{approx.}\) $P$- primary ideals $\forall i \in \{1,2,...,n\}$, then $r(W_i) = P$. From ~\ref{p-primary}. So, $r(W) = P$.\newline \newline
Now, I have to show that $W$ is an  \(\textit{approx.}\) primary ideal, let $sk$ be an element of $\Phi^{*}(W)$ and $s \notin W $ where $s,k \in R$, since $\Phi^{*}W = \bigcap_{i=1}^{n} \Phi^{*}W_i$. Thus, $sk$ $\in$ $\Phi^{*}W_i$, $\forall i \in \{1,2,...,n\}$. As $W_i$ are \(\textit{approx.}\) primary ideals, $\forall i \in \{1,2,...,n\}$. Then, for some $j \in N$,  $s \notin$ $W_j$, then $k^m$ in $W_j$ for some $m \in N$. Therefore, $k \in r(W_j)=P=r(W)$. It follows that $W$ is an \(\textit{approx.}\) primary ideal  
\end{proof}
	
For an \(\textit{approx.}\)  commutative ring ${\mathcal{R}}$ with a unity and the pair \colorbox{green!20} {$(\Phi^{*}\mathcal{R},+),(\Phi^{*}\mathcal{R},\cdot)$} are groupoids. Let $W$ be an \(\textit{approx.}\) ideal of ${\mathcal{R}}$ and $s \in \mathcal{R}$. The set $W:s$ 

\begin{center}
\colorbox{green!20} 
{$ W:s= \{ l \in \mathcal{R} \  :  sl \in W\}$}
\end{center}
is an \(\textit{approx.}\)   ideal of ${\mathcal{R}}$.

\begin{proposition}
Let $W$ be an \(\textit{approx.}\)   ideal of an \(\textit{approx.}\) commutative ring ${\mathcal{R}}$ with $s \in \mathcal{R}$, then $W:s$ is an \(\textit{approx.}\)   ideal.
\end{proposition}

\begin{proof}
For $r \in \mathcal{R}$ and $l \in W:s$, then $sl \in W$. Now, $r(sl) \in \Phi^{*}W$ since $W$ is an \(\textit{approx.}\)  ideal. Thus, $s(rl) \in \Phi^{*}W$. I can conclude that $rl \in \Phi^{*}$ $(W:s)$. \newline

Now, let $a,b \in W:s$. I have $sa$ and $sb$ are elements of $W$. Thus, $sa+sb \in \Phi^{*}W$ since $W$ is an \(\textit{approx.}\)  ideal. Furthermore, $a+b \in \Phi^{*} (W:s)$. Moreover, for any element $y \in W:s$, then $-y$ is also an element of $W:s$.
\end{proof}

\begin{theorem}
 Suppose that $W_1,W_2,...W_n$ are  \(\textit{approx.}\)   ideals of an  \(\textit{approx.}\) ring ${\mathcal{R}}$ with $s \in R$, then \[( {\bigcap_{i=1}^{n}} W_i):s = \bigcap_{i=1}^{n} (W_i :s)\].
\end{theorem}

\begin{proof}
 Let $k \in \bigcap_{i=1}^{n} (W_i :s)$, then $sk \in W_i$, $\forall i \in \{1,2,...,n\}$, which means that $sk \in \bigcap_{i=1}^{n} W_i$. Hence, $k \in (\bigcap_{i=1}^{n} W_i):s$.\newline
 
 On the other hand, let $k \in (\bigcap_{i=1}^{n} W_i):s$, thus $sk \in (\bigcap_{i=1}^{n} W_i)$. It follows that $sk \in W_i$, $\forall i \in \{1,2,...,n\}$. So, $k \in \bigcap_{i=1}^{n} (W_i :s)$.\newline
\end{proof}

\subsection{Approximately semi-primary ideals} 
\leavevmode\vspace{1em}

Another type of \(\textit{approx.}\) ideal is the \(\textit{approx.}\) semi-primary ideal as defined below:

\begin{definition}    
An \(\textit{approx.}\) ideal \( O \) of an \(\textit{approx.}\) ring \( \mathcal{R} \) is called an \(\textit{approx.}\) semi-primary ideal, provided that its radical \( r(O) \) is \(\textit{approx.}\) prime.
\end{definition}
\begin{proposition}
An \(\textit{approx.}\) ring ${\mathcal{R}}$ is said to satisfy the following property: every \(\textit{approx.}\) semi-primary ideal $O$ is an \(\textit{approx.}\) primary ideal. Thus, the following statements hold: \newline
\begin{enumerate}
\item[1.] for any \(\textit{approx.}\) proper ideal $O$ of ${\mathcal{R}}$, the quotient ring $\mathcal{R}/O$ satisfies the given property.\newline

\item[2.] Given that $Q$ and $W$ are \(\textit{approx.}\) ideals of ${\mathcal{R}}$ with $Q \subseteq W \subseteq r(Q)$, and if $Q$ is an \(\textit{approx.}\) r(Q)-primary ideal, then $W$ is also an \(\textit{approx.}\) r(Q)-primary ideal. 
\end{enumerate}
\end{proposition}
\newpage
\begin{proof}$\mbox{}$
\begin{enumerate}
\item[1.] Let $I \subseteq \mathcal{R}$ be an \(\textit{approx.}\) proper ideal of ${\mathcal{R}}$. Thus, $I/O$ is an \(\textit{approx.}\) proper ideal of $\mathcal{R}/O$. Assume that $I/O$ is an \(\textit{approx.}\) semi-primary ideal, I have to show that $I/O$ is an \(\textit{approx.}\) primary ideal.\newline

Let $(x+O),(y+O) \in \mathcal{R}/O$ such that $(x+O)(y+O)=(xy+O) \in \Phi^{*}(I)/O$, provided $xy \in \Phi^{*}I$. Since $I/O$ is an \(\textit{approx.}\) semi-primary ideal, then I can conclude that $I$ is an \(\textit{approx.}\) semi-primary ideal. However, ${\mathcal{R}}$ satisfies the given property, so  $I$ is an \(\textit{approx.}\) primary ideal. Thus, either $x \in I$ or for some $n \in N$, $y^n \in I$. I have $x+O \in I/O$ or $y^n+O \in I/O$. Moreover, $I/O$ is an \(\textit{approx.}\) primary ideal.\newline

\item[2.] Since $Q \subseteq W$, then $r(Q) \subseteq r(W)$. Moreover, $r(W)=r(Q)$. I have $Q$ is an \(\textit{approx.}\) $r(Q)$-primary ideal, provided $Q$ is an \(\textit{approx.}\) primary ideal and $r(Q)$ is an \(\textit{approx.}\) prime ideal. Thus, $Q$ is an \(\textit{approx.}\) semi-primary ideal. It follows that $r(W)$ is an \(\textit{approx.}\) prime ideal. So, I can say that $W$ is an \(\textit{approx.}\) primary  ideal. Furthermore, $W$ is an \(\textit{approx.}\)  $r(Q)$ primary ideal.  
\end{enumerate}
\end{proof}
		
\subsection{Approximately 1-absorbing primary ideals}		
\begin{definition}

Suppose that $Q$ is an \(\textit{approx.}\) ideal of an \(\textit{approx.}\)  ring ${\mathcal{R}}$. Then, $Q$ is called an  \(\textit{approx.}\) 1-absorbing primary ideal of ${\mathcal{R}}$, provided for all non-unit elements $a,b,c \in \mathcal{R}$ and $abc \in \Phi^{*}Q$, implies $ab \in Q$ or $c \in r(Q)$.
\end{definition}
		
The relation between an \(\textit{approx.}\) primary ideal and an \(\textit{approx.}\) 1- absorbing primary ideal is given by the following theorem. 

\begin{theorem}
If $Q$ is an \(\textit{approx.}\) primary ideal of ${\mathcal{R}}$, then $Q$ is an \(\textit{approx.}\) 1- absorbing primary ideal of ${\mathcal{R}}$. 
\end{theorem}
		
\begin{proof}
Let $a,b,c$ be non-unit elements of ${\mathcal{R}}$ such that $abc \in \Phi^{*}Q$. Assume that $ab = h$, then $hc \in \Phi^{*}(Q)$. However, $Q$ is an \(\textit{approx.}\) primary ideal. So, either $h \in Q$ or $c^n \in Q$ for some $n \in N$. Thus, either $ab \in Q$ or $c \in r(Q)$. Hence, I can conclude that $Q$ is an \(\textit{approx.}\) 1-  absorbing primary ideal. 
\end{proof}

\begin{theorem}
Suppose that $Q$ is an \(\textit{approx.}\) 1-  absorbing primary ideal of an \(\textit{approx.}\) ring ${\mathcal{R}}$. Then, $r(Q)$ is an \(\textit{approx.}\) prime ideal of ${\mathcal{R}}$.  
\end{theorem}

\begin{proof} 
Let $a,b$ be non-unit elements of ${\mathcal{R}}$ such that $ab \in \Phi^{*}(r(Q))$. Then for some $n \in N$, it follows that $(ab)^n = a^n.b^n$ is descriptively near to $Q$. It is sufficient to consider the case where $n$ is even, and the odd case follows similarly. Set $n = 2m$ for some $m \in N$. Hence, $a^m.a^m.b^{2m} \in \Phi^{*}Q$. However, $Q$ is an \(\textit{approx.}\) 1-  absorbing primary ideal. It follows that, either $a^m . a^m \in Q$ or $b^{nk} \in Q$, for some $k \in N$. Hence, $a \in r(Q)$ or $b \in r(Q)$. Furthermore, $r(Q)$ is an \(\textit{approx.}\)  prime ideal. 
\end{proof}

\begin{definition} \cite{maram_first_article}
 If $x$ is an element of an \(\textit{approx.}\) commutative ring ${\mathcal{R}}$, then $x$ is called an \(\textit{approx.}\) irreducible element, provided $x$ is a non-unit element and if $x= y \cdot z$, where $yz \in \Phi^{*} \mathcal{R}$, implies that either $y$ or $z$ is a unit element.
\end{definition}

\begin{theorem}
Suppose that $W$ is an \(\textit{approx.}\) 1- absorbing primary ideal of ${\mathcal{R}}$ that is not a primary ideal of ${\mathcal{R}}$. Then, 
if $ab \in \Phi^{*}W$ for some non-unit elements $a,b \in \mathcal{R}$ such that neither $a \in W$ nor $b \in r(W)$, then $a$ is an \(\textit{approx.}\) irreducible element of ${\mathcal{R}}$.
\end{theorem}

\begin{proof}
Suppose that $W$ is not an \(\textit{approx.}\)  primary ideal, but it is an \(\textit{approx.}\)  1- absorbing primary ideal. Then, there exist non-unit elements $a,b \in \mathcal{R}$ such that $ab \in \Phi^{*}(W)$, while neither $a \in W$ nor $b \in r(W)$. Assume that $a$ is not an \(\textit{approx.}\) irreducible element. Hence, $a = c.d$ for some non-unit elements $c,d$ of ${\mathcal{R}}$ with $cd \in \Phi^{*}(\mathcal{R})$. As $ab = cdb \in \Phi^{*}(W)$ and $W$ is an \(\textit{approx.}\) 1-absorbing primary ideal. Since $b \notin r(W)$, then $cd = a \in W$, which is a contradiction. Thus, $a$ is an \(\textit{approx.}\)  irreducible element.
\end{proof}

\begin{theorem}
Suppose that $\mathcal{R} = \mathcal{R}_1\times \mathcal{R}_2$, where $\mathcal{R}_1$ and $\mathcal{R}_2$ are two \(\textit{approx.}\) commutative rings with unity and $\Phi^{*} (\mathcal{R}) = \Phi^{*} \mathcal{R}_1 \times \Phi^{*}\mathcal{R}_2$. If $M= N \times \mathcal{R}_2$ is an \(\textit{approx.}\) primary ideal of ${\mathcal{\mathcal{R}}}$, then $N$ is an \(\textit{approx.}\) primary ideal of $\mathcal{R}_1$. 
\end{theorem}

\begin{proof}
Let $a,b$ be elements of  $\mathcal{R}_1$ such that $ab \in \Phi^{*}N$. Then, for $i,j \in \mathcal{R}_2$ $(a,i).(b,j) \in \Phi^{*}M$ such that $(ab,ij) \in \Phi^{*} N \times \Phi^{*}\mathcal{R}_2$. As $M$ is an \(\textit{approx.}\) primary ideal. It follows that either $(a,i) \in M$ or $(b,j)^{n} \in M$, for some $n \in N$, provided $a \in N$ or $b \in r(N)$. Hence, $N$ is an \(\textit{approx.}\) primary ideal. 
\end{proof}

\bibliographystyle{plain}

\end{document}